\documentclass[12pt]{article}
\usepackage{amsfonts, amssymb, latexsym, amscd, amsmath, amsthm, amsfonts}

\usepackage{epsfig}
\usepackage{amsmath,amsthm,amsfonts,amssymb,amscd}

\newtheorem{theorem}{Theorem}[section]

\newtheorem{proposition}[theorem]{Proposition}

\def\eps{{\varepsilon}}

\begin{document}

\title{On the spectral characterization\\ of mixed extensions of $P_{3}$}
\author{Willem H. Haemers\\[-3pt]
{\small Dept. of Econometrics and O.R.,}\\[-3pt]
{\small Tilburg University, Tilburg, The Netherlands}\\[-3pt]
{\small \tt{haemers@uvt.nl}}
\\[7pt]
Sezer Sorgun \qquad Hatice Topcu\\[-3pt]
{\small Dept. of Mathematics,}\\[-3pt]
{\small Nev\c{s}ehir Hac{\i} Bekta\c{s} Veli University,}
{\small Nev\c{s}ehir, Turkey}\\[-3pt]
{\small \tt{srgnrzs,haticekamittopcu@gmail.com}}
}
\date{}

\maketitle

\begin{abstract}
\noindent
A mixed extension of a graph $G$ is a graph $H$ obtained from $G$ by replacing each vertex of $G$ by a clique or a coclique, whilst
two vertices in $H$ corresponding to distinct vertices $x$ and $y$ of $G$ are adjacent whenever $x$ and $y$ are adjacent in $G$.
If $G$ is the path $P_3$, then $H$ has at most three adjacency eigenvalues unequal to $0$ and $-1$.
Recently, the first author classified the graphs with the mentioned eigenvalue property.
Using this classification we investigate mixed extension of $P_3$ on being determined by the adjacency spectrum.
We present several cospectral families, and with the help of a computer we find all graphs on at most $25$ vertices that
are cospectral with a mixed extension of $P_3$.
\end{abstract}

\section{Introduction}

Characterizations of graphs by means of the spectrum of the adjacency matrix is a well-studied subject.
Although it is conjectured that almost all graphs are determined by the spectrum of the adjacency matrix,
there are still relatively few graphs known to be determined by its spectrum.
The reason is that in general this property is hard to prove.
Some of these proofs are based on the classification of graphs with certain spectral properties,
such as classifications in terms of the smallest eigenvalue.
Recently the first author classified the graphs with all but at most three eigenvalues equal to $0$ or $-1$; see~\cite{H}.
Here this classification is applied to mixed extensions of the path $P_3$ (see next section), which have the mentioned spectral property.
We give the characteristic polynomial of all graphs in the classification,
and completely determine the mixed extensions of $P_3$ which are determined by the spectrum on at most 25 vertices.
Also we present several infinite families of cospectral graphs with all but three eigenvalues equal to $0$ or $-1$.

\section{Mixed extensions of $P_3$}

Let $G$ be a graph with vertex set $V(G)=\{1,\ldots,m\}$ and let $V_{1}, \ldots, V_{m}$ be mutually disjoint nonempty finite sets.
A graph $H$ with vertex set $V(H)=V_{1}{\cup}\ldots{\cup}V_{m}$ is defined as follows.
For each $i\in\{1, \ldots, m\}$, all vertices of $V_{i}$ are either mutually adjacent (form a clique),
or mutually nonadjacent (form a coclique).
For any $u\in V_{i}$ and $v\in V_{j}$ ($i\neq j$) 
$\{u,v\}$ in an edge in $H$ if and only if $\{i,j\}$ is an edge in $G$.
The graph $H$ is called a {\em mixed extension} of $G$.
A mixed extension is represented by an $m$-tuple $(t_{1}, \ldots, t_{m})$ of nonzero integers,
where $t_{i}>0$ indicates that $V_{i}$ is a clique of order $t_{i}$ and $t_{i}<0$ means that $V_{i}$ is a coclique
of order $-t_{i}$.
A mixed extension of $G$ is a special case of a {\em generalized composition} or {\em $G$-join}, introduced in
\cite{CFMR} and \cite{S}, respectively.
We refer to \cite{H}, \cite{CFMR} and \cite{S} for basic results on mixed extensions,
and to \cite{BH} or \cite{CRS} for graph spectra.

Suppose $H$ is a mixed extension of the path $P_3$ of type $(t_1,t_2,t_3)$.
Then the adjacency matrix of $H$ admits the following structure.
(As usual, $J$ is an all-ones matrix, $J_n$ is the $n\times n$ all-ones matrix, and $I_n$ is the identity matrix of order $n$.)
\[
A=\left[
\begin{array}{ccc}
\eps_1 (J_{|t_1|}-I_{|t_1|}) & J & O\\
J & \eps_2(J_{|t_2|}-I_{|t_2|}) & J\\
O & J & \eps_3(J_{|t_3|}-I_{|t_3|})
\end{array}
\right],
\]
where $\eps_i=1$ if $t_i>1$ and $\eps_i=0$ otherwise ($i=1,2,3$).

The given partition of $A$ is equitable, therefore $A$ has two kinds of eigenvalues:
the ones that have eigenvectors in the span $V$ of the characteristic vectors of the partition,
and those whose eigenvectors are orthogonal to $V$.
The first kind coincide with the eigenvalues of the quotient matrix
 \[
Q=\left[
\begin{array}{ccc}
\eps_1 (|t_1|-1) & |t_2| & 0\\
|t_1| & \eps_2(|t_2|-1) & |t_3|\\
0 & |t_2| & \eps_3(|t_3|-1)
\end{array}
\right].
\]
The second kind of eigenvalues of $H$ remain eigenvalues if we subtract an all-one block from each nonzero block of $A$.
So these eigenvalues are also eigenvalues of
\[
B=\left[
\begin{array}{ccc}
-\eps_1 I_{|t_1|} & O & O\\
O & -\eps_2 I_{|t_2|} & O\\
O & O & -\eps_3 I_{|t_3|}
\end{array}
\right],
\]
which are clearly all equal to $0$ or $-1$.
So, if $n=|t_1|+|t_2|+|t_3|$ is the order of $H$, and $q(x)=x^3-bx^2-cx+d$ is the characteristic polynomial of $Q$,
then the characteristic polynomial of $A$ equals $p(x)=q(x)(x+1)^b x^{n-b-3}$.
Indeed, the exponent of $(x+1)$ equals $b$, because the coefficient of $x^{n-1}$ in $p(x)$ equals trace$(A)=0$.
Note that $d=-\det(Q)\geq 0$, and $\det(Q)=0$ if and only if $\eps_1=\eps_3=0$
that is, both end vertices of $P_3$ are replaced by a coclique.
If this is the case then we call it an {\em improper} mixed extension of $P_3$,
because it is in fact a mixed extension of $P_2$.
If $\det(Q)\neq 0$, then the mixed extension is {\em proper}.
\begin{proposition}\label{mep3}
A proper mixed extension of the path $P_3$ has exactly two positive eigenvalues and one eigenvalue smaller than $-1$.
\end{proposition}
\begin{proof}
The quotient matrix $Q$ has at least one positive eigenvalue, and $\det(Q)<0$ gives that $Q$ (and $A$) has exactly two positive
eigenvalues.
It is well-known that a graph with smallest adjacency eigenvalue at least $-1$ is the disjoint union of complete graphs.
Therefore the smallest eigenvalue of $A$ (and $Q$) is less than $-1$.
\end{proof}

\section{Spectral characterizations}

Some known results on spectral characterizations of graphs deal with special cases of mixed extensions of $P_3$.
This includes the pineapple graphs (type $(p,1,-q)$, $p,q>0$),
and the complete graphs from which the edges of a complete bipartite subgraph are deleted (type $(p,q,r)$, $p,q,r > 0$).
The latter graphs are determined by their spectrum for all $p,q,r>0$; see~\cite{CH}.
For the pineapple graphs it is known for which $p$ and $q$ the graphs are determined by its spectrum.
For the precise conditions on $p$ and $q$ we refer to \cite{TSH'}.
It follows that among the connected graphs the pineapple graph is determined by its spectrum.
\\
An improper mixed extension of $P_3$ is either a complete bipartite graph $K_{p,q}$ ($p\geq 2, q\geq 1$),
or a complete split graph $CS_{p,q}$ ($p,q\geq 2$),
which is a complete graph $K_{p+q}$ from which the edges of a complete subgraph $K_q$ are deleted.
For these classes of graphs the spectral characterization is straightforward and known.
The complete split graph $CS_{p,q}$ is determined by its spectrum for all $p,q\geq 2$.
The complete bipartite graph $K_{p,q}$ is not determined by its adjacency spectrum if and only if $pq$
has a divisor $r$ strictly between $p$ and $q$.
Then $K_{r,pq/r}$ extended with $p+q-r-pq/r$ isolated vertices is cospectral with $K_{p,q}$.
From now on we restrict to proper mixed extensions of $P_3$.

 \begin{table}
 \begin{tabular}{|c|c|c|c|}
  \hline
    & b & c & d\\\hline
    $(-p,-q,r)$& $r-1$& $pq+qr$ & $pq(r-1)$ \\\hline
    $(-p,q,r)$& $q+r-2$& $q+r+pq-1$ & $pq(r-1)$ \\\hline
    $(p,-q,r)$& $p+r-2$& $q(p+r)+(p-1)(1-r)$ & $qr(p-1)+pq(r-1)$ \\\hline
    $(p,q,r)$& $p+q+r-3$& $2q+2r+2p-pr-3$ & $qpr+pr-p-q-r+1$ \\\hline
    $(-2,q,r,-2)$& $q+r-1$& $2q+2r$ & $4qr$ \\\hline
     $(-3,q,-2,s)$& $q+s-1$& $2s+5q-qs$ & $6qs$ \\\hline
    $(-2,-2,-3,s)$& $s$& $2s+10$ & $12s$ \\\hline
    $(5,2,-r,4)$& $8$& $6r-9$ & $34r+18$ \\\hline
    $(4,2,-r,6)$& $9$& $8r-15$ & $40r+25$ \\\hline
    $(7,2,-r,3)$& $9$& $5r-6$ & $37r+16$ \\\hline
    $(3,3,-r,6)$& $9$& $9r-15$ & $45r+25$ \\\hline
    $(4,3,-r,3)$& $7$& $6r-4$ & $30r+12$ \\\hline
    $(7,3,-r,2)$& $9$& $5r+1$ & $37r+9$ \\\hline
    $(3,4,-r,4)$& $8$& $8r-9$ & $40r+18$ \\\hline
    $(3,6,-r,3)$& $9$& $9r-6$ & $45r+16$ \\\hline
    $(4,6,-r,2)$& $9$& $8r+1$ & $40r+9$ \\\hline
    $(5,4,-r,2)$& $8$& $6r+1$ & $34r+8$ \\\hline
    $(2,2,2,7)$& $9$& $1$ & $65$ \\\hline
    $(2,2,6,3)$& $9$& $9$ & $73$ \\\hline
    $(2,2,3,4)$& $7$& $5$ & $51$ \\\hline
    $(2,3,2,5)$& $8$& $1$ & $68$ \\\hline
    $(2,3,4,3)$& $8$& $7$ & $74$ \\\hline
    $(2,5,2,4)$& $9$& $1$ & $89$ \\\hline
    $(3,2,2,3)$& $6$& $3$ & $40$ \\\hline
    $(2,5,3,3)$& $9$& $6$ & $94$ \\\hline
    $(1,p,-q,r,1)$& $p+r-1$& $(q+1)(p+r)-pr$ & $pr(2q+1)$ \\\hline
    $K_{p}+K_{q,r}$& $p-1$& $qr$ & $qr(p-1)$ \\\hline
    $K_{p}+CS_{q,r}$& $p+q-2$& $qr-(p-1)(q-1)$ & $qr(p-1)$ \\\hline
\end{tabular}
\caption{Coefficients $b,c,d$ of the characteristic polynomials of graphs in $\mathcal G''$}\label{bcd}
\end{table}

We define $\mathcal{G}$ to be the set of graphs with all but at most three adjacency eigenvalues equal to $-1$ or $0$.
Note that a graph $G$ in $\mathcal{G}$ remains in $\mathcal{G}$ if isolated vertices are added or deleted.
Therefore, for the classification, we may restrict to the set of graphs in $\mathcal{G}$ with no isolated vertices.
By $\mathcal{G}''$ we denote the class of graphs in $\mathcal{G}$ with no isolated vertices, with two positive
eigenvalues and with one eigenvalue less than $-1$.
By Proposition~\ref{mep3} we know that a proper mixed extension of $P_3$ is in $\mathcal{G''}$.
The characterization of graphs in $\mathcal G$, mentioned in the introduction, when restricted to $\mathcal G''$,
is given by the following two theorems.

\begin{theorem}
A disconnected graph $H$ belongs to $\mathcal{G''}$, if and only if $H$ is one of the following.
\\
(i) $K_{p}+K_{q,r}$ with $p,q,r \geq 2$,
\\
(ii) $K_{p}+CS_{q,r}$ with $p,q\geq 2$, $r\geq 1$.
\end{theorem}

\begin{theorem}
A connected graph $H$ belongs to $\mathcal{G''}$ if and only if $H$ is one of the following.
\\
(i) A mixed extension of $P_{3}$ of type $(-p,-q,r)$; $(-p,q,r)$; $(p,-q,r)$, or $(p,q,r)$
with $p,q\geq 1$ and $r\geq 2$,
\\
(ii) a mixed extension of $P_4$ of type $(-2,q,r,-2)$; $(-3,q,-2,s)$, or $(-2,-2,-3,s)$
with $q,r,s\geq 1$,
\\
(iii) a mixed extension of $P_4$ of type $(p,q,-r,s)$ with $r\geq 1$ and $(p,q,s)\in
\big\{(5,2,4), \allowbreak (4,2,6),(7,2,3),(3,3,6),(4,3,3),(7,3,2),(3,4,4),(3,6,3),(4,6,2),(5,4,2)\big\}$,
\\
(iv) a mixed extension of $P_4$ of type $(p,q,r,s)$, with $(p,q,r,s)\in
\big\{(2,2,2,7), (2,2,6,3),\allowbreak(2,2,3,4), (2,3,2,5), (2,3,4,3), (2,5,2,4),(2,5,3,3),(3,2,2,3)\big\}$,
\\
(v) a mixed extension of $P_5$ of type $(1,p,-q,r,1)$ with $p,q,r \geq 1$.
\end{theorem}

The characteristic polynomial $p(x)$ of a graph in $\mathcal{G''}$ with $n$ vertices can be written as:
\[
p(x)=(x^3-bx^2-cx+d)(x+1)^bx^{n-b-3}, \mbox{ with } b\geq 0 \mbox{ and } d>0
\]
(note that the multiplicity of the eigenvalue $-1$ equals $b$, because trace$(A)=0$).
So the nonzero part of the spectrum of a graph in $\mathcal{G''}$ is determined by the coefficients $b$, $c$ and $d$.
Table~\ref{bcd} gives these coefficients for each type of the above classification.
If two graphs in the classification have the same $b$,$c$ and $d$, then the nonzero part of the spectrum is the same.
We will call such a pair {\em pseudo-cospectral}.
If two pseudo-cospectral graphs have different order, we can extend the smaller graph (or both graphs) with some isolated vertices, so that
the two graphs become cospectral.
Therefore Table~\ref{bcd} gives all information needed to decide which proper mixed extensions of $P_3$ are determined by their spectrum.
Nevertheless, it is far too complicated to give a general result.
Therefore we present some special cases.

\begin{theorem}
{\rm \cite{CH}} Suppose $H$ is a mixed extension of $P_3$ of order $n$ and type $(p,q,r)$ with $p,q,r \geq 1$.
Then $H$ is determined by the spectrum of its adjacency matrix.
\end{theorem}

\begin{proof}
From Table~\ref{bcd} we see that for this case the coefficient $b$ equals $p+q+r-3=n-3$.
For every graph $H'$ of order $n'$ pseudo-cospectral with $H$, which belongs to one of the other types, we have $b<n'-3$.
So $H'$ has more vertices than $H$, and we cannot obtain a graph cospectral with $H$ by adding isolated vertices to $H'$.
If $H'$ is a mixed extension of $P_3$ of type $(p',q',r')$ with $p',q',r'\geq 1$,
and $H'$ is cospectral with $H$, then $H$ and $H'$ have the same coefficients $b$, $c$ an $d$,
which implies $p=p'$, $q=q'$, and $r=r'$.
\end{proof}

The following results follow straightforwardly from Table~\ref{bcd}.

\begin{proposition}\label{3types}
The following types of mixed extensions of $P_3$ are pseudo-cospectral with a non-isomorphic graph in $\mathcal G''$.
\\
(i) Type $(p,-q,p)$ with $K_{p}+CS_{p,2q}$ where $p\geq 2,q\geq1$,
\\
(ii) type $(p,(p-1)(p-2),p)$ with $K_{p(p-1)} + K_{p-1,p-1}$, where $p\geq 3$,
\\
(iii) type $(p,q,p)$ with $K_p + CS_{p+q-1,r}$ where $r=1+pq/(p+q-1)$ is an integer, and $p\geq 2$, $q\geq 1$.
\end{proposition}

Notice that this proposition does not give any graph which is cospectral and non-isomorphic with a mixed extension of $P_3$,
because in all three cases the second graph has more vertices than the first one.
For the pineapple graph $K_p^q$, which is a mixed extension of $P_3$ of type $(p-1,1,-q)$, we find (see~\cite{TSH}).

\begin{proposition}\label{pine}
The pineapple graph $K_{2p}^{p^2}$ is cospectral with the mixed extension of $P_3$ of type $(p,-p,p)$ extended with $p(p-1)$
isolated vertices.
\end{proposition}

See \cite{TSH'} for more examples of graphs non-isomorphic but cospectral with $K_p^q$.
As remarked before, there exists no connected example.
However, there do exist connected non-isomorphic cospectral mixed extensions of $P_3$.

\begin{proposition}\label{cosp3}
The mixed extensions of $P_3$ of types $(p,-q,q(2q-p-1)/(q-p))$ and $(-q,2q-1,p(2q-p-1)/(q-p))$
are cospectral whenever $q(2q-p-1)/(q-p)$ is a positive integer, and $p,q\geq 1$.
\end{proposition}

Note that there are infinitely many values of $p$ and $q$ for which the above fraction is integral.
For example if $p=1$ or $p=q-1$.

\section{Enumeration}

Using Table~\ref{bcd}, we generated by computer (using Maple) a list of all non-isomorphic graphs in $\mathcal G''$ on at
most $25$ vertices.
The list contains almost 10000 graphs. 
We ordered the graphs lexicographically with respect to the coefficients $b,c,d$.
Then the pseudo-cospectral graphs in $\mathcal G''$ become consecutive items in the table with the same $b,c,d$.
Since the list is very long we only consider the mixed extensions of $P_3$ that have at least one non-isomorphic pseudo-cospectral mate
in the list.
By use of the shortened list we found the pseudo-cospectral examples of Propositions~\ref{3types} and \ref{cosp3}.
Next we deleted the cases given in Proposition~\ref{3types}, \ref{pine} and \ref{cosp3} from the list.
The final list is given in Table~\ref{25}.
Thus this table together with Propositions~\ref{3types} to \ref{cosp3} give all mixed extensions of $P_3$ of order $n\leq 25$
for which there exist at least one pseudo-cospectral graph in $\mathcal G''$ of order at most $25$.
Since Proposition~\ref{3types} gives no graphs cospectral with a mixed extension of $P_3$, we can conclude the following:

\begin{theorem}
Suppose $H$ is a proper mixed extension of $P_3$ of order $n\leq 25$.
Then $H$ is determined by the spectrum of the adjacency matrix if and only if
$H$ is not one of the graphs given in Propositions~\ref{pine} or \ref{cosp3},
and every graph in Table~\ref{25} pseudo-cospectral with $H$ has more vertices than $H$.
\end{theorem}

\subsection*{Acknowledgments}

We thank one referee for the correction of some nasty errors,
and another referee for pointing at the references \cite{CFMR} and \cite{S}.

This work is partially supported by TUBITAK (the Scientific and Technological Research Council of Turkey) research project 117F489.
\\
\begin{table}
\caption{Graphs in $\mathcal G''$ of order $n\leq 25$ pseudo-cospectral with a mixed extension of $P_3$} \label{25}
{\scriptsize
\begin{tabular}{ll}
\\
\begin{tabular}{|rrr|c|rrrr|r|}
\hline
$b$ &$c$ &$d$    &Type           & $p$ &$q$  &$r$&$s$&$n$ \\ \hline
   2 & 10 &  14  &$(-p,q,r)$     &  7 & 1 & 3 &   & 11 \\
   2 & 10 &  14  &$(1,p,-q,r,1)$ &  1 & 3 & 2 &   &  8 \\ \hline
   3 &  8 &  12  &$(-p,q,r)$     &  4 & 1 & 4 &   &  9 \\
   3 &  8 &  12  &$(-p,q,r,-s)$  &  2 & 1 & 3 & 2 &  8 \\ \hline
   3 &  9 &  15  &$(-p,q,r)$     &  5 & 1 & 4 &   & 10 \\
   3 &  9 &  15  &$(1,p,-q,r,1)$ &  1 & 2 & 3 &   &  8 \\ \hline
   3 & 10 &  12  &$(-p,q,r)$     &  3 & 2 & 3 &   &  8 \\
   3 & 10 &  12  &$Kp+CSq,r$     &  2 & 3 & 4 &   &  9 \\ \hline
   3 & 14 &  18  &$(-p,-q,r)$    & 3   &    2   &    4  &        &     9     \\
   3 & 14 &  18  &$(-p,q,-r,s)$  & 3   &    3   &    2  &     1  &     9     \\ \hline
   3 & 16 &  24  &$(-p,-q,r)$    & 4   &    2   &    4  &        &     10    \\
   3 & 16 &  24  &$(-p,q,r)$     & 6   &    2   &    3  &        &     11    \\  \hline
   3 & 16 &  36  &$(-p,q,r)$     & 12  &    1   &    4  &        &     17    \\
   3 & 16 &  36  &$(-p,-q,-r,s)$ & 2   &    2   &    3  &     3  &     10    \\
   3 & 16 &  36  &$(1,p,-q,r,1)$ & 2   &    4   &    2  &        &     10    \\
   3 & 16 &  36  &$K_p+CS_{q,r}$ & 3   &    2   &    9  &        &     14    \\  \hline
   3 & 18 &  28  &$(-p,q,r)$     & 7   &    2   &    3  &        &     12    \\
   3 & 18 &  28  &$(p,-q,r)$     & 2   &    4   &    3  &        &     9    \\   \hline
   3 & 28 &  60  &$(-p,-q,r)$    & 10  &    2   &    4  &        &     16   \\
   3 & 28 &  60  &$(1,p,-q,r,1)$ & 2   &    7   &    2  &        &     13   \\
   3 & 28 &  60  &$K_p+CS_{q,r}$ & 3   &    2   &    15 &        &     20   \\   \hline
   3 & 40 &  72  &$(-p,-q,r)    $& 6   &    4   &    4  &        &     14   \\
   3 & 40 &  72  &$(-p,q,r)     $& 18  &    2   &    3  &        &     23   \\   \hline
   3 & 52 &  108 &$(-p,-q,r)    $& 9   &    4   &    4  &        &     17   \\
   3 & 52 &  108 &$(1,p,-q,r,1) $& 2   &    13  &    2  &        &     19   \\  \hline
   4 &  9 &  12  &$(-p,q,r)$     & 2   &    2   &    4  &        &     8    \\
   4 &  9 &  12  &$K_p+CS_{q,r}$ & 2   &    4   &    3  &        &     9    \\   \hline
   4 &  15 & 30  &$(-p,q,r)$     & 5   &    2   &    4  &        &     11   \\
   4 &  15 & 30  &$(p,-q,r)$     & 2   &    3   &    4  &        &     9    \\    \hline
   4 &  21 & 24  &$(-p,-q,r)$    & 2   &    3   &    5  &        &     10   \\
   4 &  21 & 24  &$K_p+CS_{q,r}$ & 2   &    4   &    6  &        &     12   \\    \hline
   4 &  23 & 54  &$(-p,q,r)$     & 9   &    2   &    4  &        &     15   \\
   4 &  23 & 54  &$K_p+CS_{q,r}$ & 3   &    3   &    9  &        &     15     \\   \hline
   4 &  39 & 102 &$(-p,q,r)$     & 17  &    2   &    4  &        &     23     \\
   4 &  39 & 102 &$(1,p,-q,r,1)$ & 2   &    8   &    3  &        &     15     \\    \hline
5 &      14 &     40  &$ (-p,q,r)      $ &  8   &    1  &     6     &      &    15 \\
5 &      14 &     40  &$ K_p + CS_{q,r}$ &  3   &    4  &     5     &      &    12 \\  \hline
5 &      16 &     20  &$ (-p,-q,r)     $ &  2   &    2  &     6     &      &    10 \\
5 &      16 &     20  &$ K_p + CS_{q,r}$ &  2   &    5  &     4     &      &    11 \\  \hline
5 &      18 &     48  &$ (-p,q,r)      $ &  6   &    2  &     5     &      &    13 \\
5 &      18 &     48  &$ K_p + CS_{q,r}$ &  3   &    4  &     6     &      &    13 \\   \hline
5 &      24 &     90  &$ (-p,q,r)      $ &  18  &    1  &     6     &      &    25 \\
5 &      24 &     90  &$ K_p + CS_{q,r}$ &  4   &    3  &     10    &      &    17 \\   \hline
5 &      28 &     88  &$ (-p,q,r)      $ &  11  &    2  &     5     &      &    18 \\
5 &      28 &     88  &$ (1,p,-q,r,1)  $ &  2   &    5  &     4     &      &    13 \\   \hline
5 &      30 &     72  &$ (-p,q,r)      $ &  8   &    3  &     4     &      &    15 \\
5 &      30 &     72  &$ K_p + CS_{q,r}$ &  3   &    4  &     9     &      &    16 \\   \hline
5 &      36 &     90  &$ (-p,-q,r)     $ &  6   &    3  &     6     &      &    15 \\
5 &      36 &     90  &$ (-p,q,r)      $ &  10  &    3  &     4     &      &    17 \\  \hline
5 &      36 &     120 &$ (-p,-q,r)     $ &  12  &    2  &     6     &      &    20 \\
5 &      36 &     120 &$ (-p,q,r)      $ &  15  &    2  &     5     &      &    22 \\  \hline
\end{tabular}
& \begin{tabular}{|rrr|c|rrrr|r|}
\hline
$b$ &  $c$   &     $d$ & Type        & $p$ &$q$&$r$&$s$&$n$ \\ \hline
 5 &      42 &     144 &$ (-p,q,r)      $ &  18  &    2  &     5     &      &    25 \\
 5 &      42 &     144 &$ K_p + CS_{q,r}$ &   4   &    3  &     16    &      &    23 \\   \hline
 5 &      57 &     153 &$ (-p,q,r)      $ &   17  &    3  &     4     &      &    24 \\
 5 &      57 &     153 &$ (p,-q,r)      $ &   3   &    9  &     4     &      &    16 \\   \hline
 5 &      70 &     200 &$ (-p,-q,r)     $ &   8   &    5  &     6     &      &    19 \\
 5 &      70 &     200 &$ (1,p,-q,r,1)  $ &   2   &    12 &     4     &      &    20 \\   \hline
 6 &      17 &     50  &$ (-p,q,r)      $ &   5   &    2  &     6     &      &    13 \\
 6 &      17 &     50  &$ K_p + CS_{q,r}$ &   3   &    5  &     5     &      &    13 \\   \hline
 6 &      19 &     48  &$ (-p,q,r)      $ &   4   &    3  &     5     &      &    12 \\
 6 &      19 &     48  &$ (p,-q,r)      $ &   2   &    3  &     6     &      &    11 \\   \hline
 6 &      19 &     60  &$ (-p,q,r)      $ &   6   &    2  &     6     &      &    14 \\
 6 &      19 &     60  &$ (-p,q,-r,s)   $ &   3   &    5  &     2     &  2   &    12 \\   \hline
 6 &      22 &     60  &$ (-p,q,r)      $ &   5   &    3  &     5     &      &    13 \\
 6 &      22 &     60  &$ K_p + CS_{q,r}$ &   3   &    5  &     6     &      &    14 \\    \hline
 6 &      25 &     90  &$ (-p,q,r)      $ &   9   &    2  &     6     &      &    17 \\
 6 &      25 &     90  &$ (1,p,-q,r,1)  $ &   2   &    4  &     5     &      &    13 \\    \hline
 6 &      28 &     84  &$ (-p,-q,r)     $ &   7   &    2  &     7     &      &    16 \\
 6 &      28 &     84  &$ (-p,q,r)      $ &   7   &    3  &     5     &      &    15 \\    \hline
 6 &      31 &     120 &$ (p,-q,r)      $ &   4   &    5  &     4     &      &    13 \\
 6 &      31 &     120 &$ (-p,q,r)      $ &   12  &    2  &     6     &      &    20 \\    \hline
 6 &      40 &     132 &$ (-p,q,r)      $ &   11  &    3  &     5     &      &    19 \\
 6 &      40 &     132 &$ (p,-q,r)      $ &   3   &    6  &     5     &      &    14 \\    \hline
 6 &      44 &     180 &$ (-p,-q,r)     $ &   15  &    2  &     7     &      &    24 \\
 6 &      44 &     180 &$ (1,p,-q,r,1)  $ &   3   &    7  &     4     &      &    16 \\    \hline
 6 &      50 &     90  &$ (-p,-q,r)     $ &   3   &    5  &     7     &      &    15 \\
 6 &      50 &     90  &$ (1,p,-q,r,1)  $ &   1   &    7  &     6     &      &    16 \\     \hline
 6 &      51 &     180 &$ (-p,-q,r)     $ &   10  &    3  &     7     &      &    20 \\
 6 &      51 &     180 &$ K_p + CS_{q,r}$ &   4   &    4  &     15    &      &    23 \\     \hline
 6 &      55 &     192 &$ (-p,q,r)      $ &   16  &    3  &     5     &      &    24 \\
 6 &      55 &     192 &$ (p,-q,r)      $ &   4   &    8  &     4     &      &    16 \\     \hline
 7  &      16 &      48  & $(-p,q,r)       $ &   4   &     2  &      7  &         &      13  \\
 7  &      16 &      48  & $(-p,q,r,-s)    $ &   2   &     2  &      6  &      2  &      12  \\   \hline
 7  &      20 &      60  & $(-p,q,r)       $ &   4   &     3  &      6  &         &      13  \\
 7  &      20 &      60  & $K_p + CS_{q,r} $ &   3   &     6  &      5  &         &      14  \\   \hline
 7  &      20 &      84  & $(-p,q,r)       $ &   12  &     1  &      8  &         &      21  \\
 7  &      20 &      84  & $(1,p,-q,r,1)   $ &   2   &     3  &      6  &         &      13  \\   \hline
 7  &      23 &      105 & $(-p,q,r)       $ &   15  &     1  &      8  &         &      24  \\
 7  &      23 &      105 & $K_p + CS_{q,r} $ &   4   &     5  &      7  &         &      16  \\   \hline
 7  &      26 &      108 & $(-p,q,r)       $ &   9   &     2  &      7  &         &      18  \\
 7  &      26 &      108 & $(p,-q,r)       $ &   3   &     4  &      6  &         &      13  \\   \hline
 7  &      28 &      120 & $(-p,q,r)       $ &   10  &     2  &      7  &         &      19  \\
 7  &      28 &      120 & $K_p + CS_{q,r} $ &   4   &     5  &      8  &         &      17  \\   \hline
 7  &      30 &      42  & $(-p,-q,r)      $ &   2   &     3  &      8  &         &      13  \\
 7  &      30 &      42  & $(-p,q,-r,s)    $ &   3   &     7  &      2  &      1  &      13  \\    \hline
 7  &      33 &      63  & $(-p,-q,r)      $ &   3   &     3  &      8  &         &      14  \\
 7  &      33 &      63  & $(1,p,-q,r,1)   $ &   1   &     4  &      7  &         &      14  \\    \hline
 7  &      35 &      135 & $(-p,q,r)       $ &   9   &     3  &      6  &         &      18  \\
 7  &      35 &      135 & $(p,-q,r)       $ &   3   &     5  &      6  &         &      14  \\    \hline
 7  &      38 &      150 & $(-p,q,r)       $ &   10  &     3  &      6  &         &      19  \\
 7  &      38 &      150 & $K_p + CS_{q,r} $ &   4   &     5  &      10 &         &      19  \\    \hline
 7  &      64 &      224 & $(-p,-q,r)      $ &   8   &     4  &      8  &         &      20  \\
 7  &      64 &      224 & $(-p,q,r)       $ &   14  &     4  &      5  &         &      23  \\    \hline
\end{tabular}
\end{tabular}
}
\end{table}
%
%
\begin{table}
{\scriptsize
\begin{tabular}{ll}
\begin{tabular}{|rrr|c|rrrr|r|}
\hline
$b$ &$c$ &$d$   &            Type            &   $p$ &     $q$&      $r$&      $s$&      $n$ \\ \hline
 7  &      68 &      240 & $(-p,q,r)       $ &   15  &     4  &      5  &         &      24  \\
 7  &      68 &      240 & $K_p + CS_{q,r} $ &   4   &     5  &      16 &         &      25  \\    \hline
 7  &      78 &      210 & $(-p,-q,r)      $ &   5   &     6  &      8  &         &      19  \\
 7  &      78 &      210 & $(-p,q,r)       $ &   14  &     5  &      4  &         &      23  \\     \hline
 8  &      8  &      64  & $(p,-q,r)       $ &   3   &     2  &      7  &         &      12  \\
 8  &      8  &      64  & $(Kp+Kq,r)      $ &   9   &     2  &      4  &         &      15  \\   \hline
 8  &      13 &      32  & $(-p,q,r)       $ &   4   &     1  &      9  &         &      14  \\
 8  &      13 &      32  & $(p,q,r)        $ &   2   &     6  &      3  &         &      11  \\   \hline
 8  &      16 &      56  & $(-p,q,r)       $ &   7   &     1  &      9  &         &      17  \\
 8  &      16 &      56  & $K_p + CS_{q,r} $ &   3   &     7  &      4  &         &      14  \\   \hline
 8  &      18 &      72  & $(-p,q,r)       $ &   9   &     1  &      9  &         &      19  \\
 8  &      18 &      72  & $(-p,q,r,-s)    $ &   2   &     3  &      6  &      2  &      13  \\   \hline
 8  &      19 &      40  & $(-p,q,r)       $ &   2   &     5  &      5  &         &      12  \\
 8  &      19 &      40  & $(1,p,-q,r,1)   $ &   1   &     2  &      8  &         &      13  \\   \hline
 8  &      27 &      126 & $(-p,q,r)       $ &   9   &     2  &      8  &         &      19  \\
 8  &      27 &      126 & $K_p + CS_{q,r} $ &   4   &     6  &      7  &         &      17  \\   \hline
 8  &      33 &      144 & $(-p,q,r)       $ &   8   &     3  &      7  &         &      18  \\
 8  &      33 &      144 & $K_p + CS_{q,r} $ &   4   &     6  &      8  &         &      18  \\    \hline
 8  &      38 &      160 & $(-p,-q,r)      $ &   10  &     2  &      9  &         &      21  \\
 8  &      38 &      160 & $(p,-q,r)       $ &   3   &     5  &      7  &         &      15  \\    \hline
 8  &      45 &      180 & $(-p,q,r)       $ &   9   &     4  &      6  &         &      19  \\
 8  &      45 &      180 & $K_p + CS_{q,r} $ &   4   &     6  &      10 &         &      20  \\    \hline
 8  &      51 &      126 & $(-p,q,r)       $ &   7   &     6  &      4  &         &      17  \\
 8  &      51 &      126 & $K_p + CS_{q,r} $ &   3   &     7  &      9  &         &      19  \\    \hline
 8  &      54 &      270 & $(-p,q,r)       $ &   15  &     3  &      7  &         &      25  \\
 8  &      54 &      270 & $(1,p,-q,r,1)   $ &   3   &     7  &      6  &         &      18  \\    \hline
 8  &      57 &      240 & $(-p,-q,r)      $ &   10  &     3  &      9  &         &      22  \\
 8  &      57 &      240 & $(-p,q,r)       $ &   12  &     4  &      6  &         &      22  \\    \hline
 8  &      68 &      256 & $(-p,-q,r)      $ &   8   &     4  &      9  &         &      21  \\
 8  &      68 &      256 & $(p,-q,r)       $ &   3   &     8  &      7  &         &      18  \\     \hline
 8  &      128&      448 & $(-p,-q,r)      $ &   7   &     8  &      9  &         &      24  \\
 8  &      128&      448 & $(p,-q,r)       $ &   3   &     14 &      7  &         &      24  \\   \hline
 9  &      3  &      61  & $(p,q,r)        $ &   3   &     3  &      6  &         &      12  \\
 9  &      3  &      61  & $(p,q,-r,s)     $ &   3   &     6  &      1  &      3  &      13  \\   \hline
 9  &      4  &      90  & $(p,-q,r)       $ &   4   &     2  &      7  &         &      13  \\
 9  &      4  &      90  & $(p,q,-r,s)     $ &   7   &     2  &      2  &      3  &      14  \\   \hline
 9  &      9  &      49  & $(p,q,r)        $ &   2   &     4  &      6  &         &      12  \\
 9  &      9  &      49  & $(p,q,-r,s)     $ &   4   &     6  &      1  &      2  &      13  \\   \hline
 9  &      15 &      135 & $(p,-q,r)       $ &   4   &     3  &      7  &         &      14  \\
 9  &      15 &      135 & $(Kp+Kq,r)      $ &   10  &     3  &      5  &         &      18  \\   \hline
 9  &      18 &      64  & $(-p,q,r)       $ &   4   &     2  &      9  &         &      15  \\
 9  &      18 &      64  & $K_p + CS_{q,r} $ &   3   &     8  &      4  &         &      15  \\   \hline
 9  &      24 &      112 & $(-p,q,r)       $ &   7   &     2  &      9  &         &      18  \\
 9  &      24 &      112 & $(1,p,-q,r,1)   $ &   2   &     3  &      8  &         &      15  \\    \hline
 9  &      24 &      126 & $(-p,q,r)       $ &   14  &     1  &      10 &         &      25  \\
 9  &      24 &      126 & $K_p + CS_{q,r} $ &   4   &     7  &      6  &         &      17  \\    \hline
 9  &      28 &      72  & $(-p,q,r)       $ &   3   &     6  &      5  &         &      14  \\
 9  &      28 &      72  & $(-p,-q,r)      $ &   4   &     2  &      10 &         &      16  \\    \hline
 9  &      31 &      63  & $(-p,q,r)       $ &   3   &     7  &      4  &         &      14  \\
 9  &      31 &      63  & $(1,p,-q,r,1)   $ &   1   &     3  &      9  &         &      15  \\    \hline
 9  &      31 &      147 & $(-p,q,r)       $ &   7   &     3  &      8  &         &      18  \\
 9  &      31 &      147 & $K_p + CS_{q,r} $ &   4   &     7  &      7  &         &      18  \\    \hline
\end{tabular}
&
\begin{tabular}{|rrr|c|rrrr|r|}
\hline
$b$ &$c$ &$d$ &          Type              &     $p$ &$q$&$r$&$s$&$n$ \\ \hline     %
9  &      34 &      96  &$(-p,q,r)       $ &     4   &     6  &      5  &         &      15  \\
9  &      34 &      96  &$K_p + CS_{q,r} $ &     3   &     8  &      6  &         &      17  \\    \hline
9  &      34 &      144 &$(-p,q,r)       $ &     6   &     4  &      7  &         &      17  \\
9  &      34 &      144 &$(1,p,-q,r,1)   $ &     2   &     4  &      8  &         &      16  \\     \hline
9  &      38 &      168 &$(-p,q,r)       $ &     7   &     4  &      7  &         &      18  \\
9  &      38 &      168 &$K_p + CS_{q,r} $ &     4   &     7  &      8  &         &      19  \\   \hline
9  &      49 &      273 &$(-p,q,r)       $ &     13  &     3  &      8  &         &      24  \\
9  &      49 &      273 &$(1,p,-q,r,1)   $ &     3   &     6  &      7  &         &      18  \\   \hline
9  &      64 &      216 &$(-p,-q,r)      $ &     6   &     4  &      10 &         &      20  \\
9  &      64 &      216 &$(-p,q,r)       $ &     9   &     6  &      5  &         &      20  \\   \hline
9  &      94 &      336 &$(-p,q,r)       $ &     14  &     6  &      5  &         &      25  \\
9  &      94 &      336 &$(1,p,-q,r,1)   $ &     2   &     10 &      8  &         &      22  \\   \hline
 10  &    20 &     72   &  $(-p,q,r)      $ & 3    &   3   &    9     &      &    15   \\
 10  &    20 &     72   &  $K_p + CS_{q,r}$ & 3    &   9   &    4     &      &    16   \\    \hline
 10  &    27 &     144  &  $(-p,q,r)      $ & 8    &   2   &    10    &      &    20   \\
 10  &    27 &     144  &  $K_p + CS_{q,r}$ & 4    &   8   &    6     &      &    18   \\    \hline
 10  &    29 &     90   &  $(-p,q,r)      $ & 3    &   6   &    6     &      &    15   \\
 10  &    29 &     90   &  $K_p + CS_{q,r}$ & 3    &   9   &    5     &      &    17   \\     \hline
 10  &    32 &     168  &  $(-p,q,r)      $ & 7    &   3   &    9     &      &    19   \\
 10  &    32 &     168  &  $(p,-q,r)      $ & 3    &   4   &    9     &      &    16   \\   \hline
 10  &    35 &     168  &  $(-p,q,r)      $ & 6    &   4   &    8     &      &    18   \\
 10  &    35 &     168  &  $K_p + CS_{q,r}$ & 4    &   8   &    7     &      &    19   \\   \hline
 10  &    59 &     240  &  $(-p,q,r)      $ & 8    &   6   &    6     &      &    20   \\
 10  &    59 &     240  &  $K_p + CS_{q,r}$ & 4    &   8   &    10    &      &    22   \\   \hline
 10  &    63 &     364  &  $(-p,q,r)      $ & 13   &   4   &    8     &      &    25   \\
 10  &    63 &     364  &  $(p,-q,r)      $ & 4    &   7   &    8     &      &    19   \\   \hline
 10  &    66 &     330  &  $(-p,-q,r)     $ &  11  &    3   &    11  &       &     25  \\
 10  &    66 &     330  &  $(-p,q,r)      $ &  11  &    5   &    7   &       &     23  \\    \hline
 11  &    -12&     108  &  $(1,p,-q,r,1)  $ &  6   &    1   &    6   &       &     15  \\
 11  &    -12&     108  &  $K_p + CS_{q,r}$ &  7   &    6   &    3   &       &     16  \\    \hline
 11  &    24 &     108  &  $(-p,q,r)      $ &  4   &    3   &    10  &       &     17  \\
 11  &    24 &     108  &  $(-p,q,r,-s)   $ &  2   &    3   &    9   &    2  &     16  \\     \hline
 11  &    30 &     162  &  $(-p,q,r)      $ &  6   &    3   &    10  &       &     19  \\
 11  &    30 &     162  &  $K_p + CS_{q,r}$ &  4   &    9   &    6   &       &     19  \\   \hline
 11  &    45 &     297  &  $(-p,q,r)      $ &  11  &    3   &    10  &       &     24  \\
 11  &    45 &     297  &  $(1,p,-q,r,1)  $ &  3   &    5   &    9   &       &     19  \\   \hline
 11  &    48 &     216  &  $(-p,q,r)      $ &  6   &    6   &    7   &       &     19  \\
 11  &    48 &     216  &  $K_p + CS_{q,r}$ &  4   &    9   &    8   &       &     21  \\   \hline
 11  &    52 &     320  &  $(-p,q,r)      $ &  10  &    4   &    9   &       &     23  \\
 11  &    52 &     320  &  $K_p + CS_{q,r}$ &  5   &    8   &    10  &       &     23  \\   \hline
 12   &   1   &    132  &  $(p,-q,r)      $ &  4   &    2   &    10   &       &    16  \\
 12   &   1   &    132  &  $K_p + CS_{q,r}$ &  12  &    2   &    6    &       &    20  \\    \hline
 12   &   19  &    66   &  $(-p,q,r)      $ &  3   &    2   &    12   &       &    17  \\
 12   &   19  &    66   &  $(p,q,r)       $ &  2   &    9   &    4    &       &    15  \\     \hline
 12   &   33  &    180  &  $(-p,q,r)      $ &  5   &    4   &    10   &       &    19  \\
 12   &   33  &    180  &  $K_p + CS_{q,r}$ &  4   &    10  &    6    &       &    20  \\   \hline
 12   &   43  &    210  &  $(-p,q,r)      $ &  5    &   6     &  8    &       & 19    \\
 12   &   43  &    210  &  $K_p + CS_{q,r}$ &  4    &   10    &  7    &       & 21    \\     \hline
 12   &   49  &    324  &  $(-p,q,r)      $ &  9    &   4     &  10   &       & 23    \\
 12   &   49  &    324  &  $K_p + CS_{q,r}$ &  5    &   9     &  9    &       & 23    \\   \hline
  12 &    57   &    396  & $(-p,q,r)      $ & 11   &   4     &    10   &    & 25    \\
  12 &      57 &    396 &$(p,-q,r)      $ &  4  &      6  &      10 &    & 20    \\    \hline

\end{tabular}
\end{tabular}
}
\end{table}
%
%
\begin{table}
{\scriptsize
\begin{tabular}{ll}
\begin{tabular}{|rrr|c|rrrr|r|}
\hline
$b$  &     $c$ &$d$      & Type             &  $p$ &$q$&$r$&$s$&$n$ \\ \hline
  12 &      58 &      360 &$(-p,q,r)      $ &  9  &      5  &      9  &    &  23   \\
  12 &      58 &      360 &$K_p + CS_{q,r}$ &  5  &      9  &      10 &    &  24   \\    \hline
  13 &      -7 &      165 &$(p,q,r)       $ &  6  &      4  &      6  &    &  16   \\
  13 &      -7 &      165 &$K_p + CS_{q,r}$ &  12 &      3  &      5  &    &  20   \\     \hline
  13 &      26 &      96  &$(-p,q,r)      $ &  2  &      6  &      9  &    &  17   \\
  13 &      26 &      96  &$K_p + CS_{q,r}$ &  3  &      12 &      4  &    &  19   \\   \hline
  13 &      54 &      360 &$(-p,q,r)      $ &  8  &      5  &      10 &    &  23   \\
  13 &      54 &      360 &$K_p + CS_{q,r}$ &  5  &      10 &      9  &    &  24   \\   \hline
  14 &      -17&      176 &$(p,q,r)       $ &  6  &      3  &      8  &    &  17   \\
  14 &      -17&      176 &$K_p + CS_{q,r}$ &  12 &      4  &      4  &    &  20   \\   \hline
  14 &      39 &      216 &$(-p,q,r)      $ &  4  &      6  &      10 &    &  20   \\
  14 &      39 &      216 &$K_p + CS_{q,r}$ &  4  &      12 &      6  &    &  22   \\    \hline
  15 &      32 &      192 &$(-p,q,r)      $ &  4  &      4  &      13 &    &  21   \\
  15 &      32 &      192 &$(-p,q,r,-s)   $ &  2  &      4  &      12 & 2  & 20    \\    \hline
  15 &      52  &     252 &$(-p,q,r)      $ &      4  &      9   &     8   &   &  21 \\
  15 &      52  &     252 &$(1,p,-q,r,1)  $ &      2  &      4   &     14  &   &  22 \\     \hline
  16 &      -37 &     198 &$(p,q,r)       $ &      8  &      2   &     9   &   &     19    \\
  16 &      -37 &     198 &$K_p + CS_{q,r}$ &      12 &      6   &     3   &   &     21    \\    \hline
  16 &      -28 &     288 &$(p,-q,r)      $ &      9  &      2   &     9   &   &     20    \\
  16 &      -28 &     288 &$K_p + CS_{q,r}$ &      9  &      9   &     4   &   &     22    \\     \hline
\end{tabular}
&
\begin{tabular}{|rrr|c|rrrr|r|}
\hline
$b$  &$c$       &$d$      &Type             & $p$ &$q$&$r$&$s$&$n$ \\ \hline
  16 &      -19 &     252 &$(p,-q,r)      $ &      6  &      2   &     12  &   &     20    \\
  16 &      -19 &     252 &$(p,q,r)       $ &      6  &      4   &     9   &   &     19    \\   \hline
  16 &      -10 &     252 &$(p,q,r)       $ &      5  &      5   &     9   &   &     19    \\
  16 &      -10 &     252 &$K_p + CS_{q,r}$ &      15 &      3   &     6   &   &     24    \\   \hline
  16 &      31  &     210 &$(-p,q,r)      $ &      7  &      2   &     16  &   &     25    \\
  16 &      31  &     210 &$K_p + CS_{q,r}$ &      4  &      14  &     5   &   &     23    \\   \hline
  16 &      38  &     210 &$(-p,q,r)      $ &      3  &      7   &     11  &   &     21    \\
  16 &      38  &     210 &$(1,p,-q,r,1)  $ &      2  &      3   &     15  &   &     22    \\   \hline
  16 &      72  &     330 &$(-p,q,r)      $ &      5  &      11  &     7   &   &     23    \\
  16 &      72  &     330 &$(1,p,-q,r,1)  $ &    2  &      5   &     15  &      &     24   \\    \hline
  17 &      33  &     225 &$(-p,q,r)      $ &    5  &      3   &     16  &      &     24   \\
  17 &      33  &     225 &$K_p + CS_{q,r}$ &    4  &      15  &     5   &      &     24   \\    \hline
  17 &      36  &     270 &$(-p,q,r)      $ &    6  &      3   &     16  &      &     25   \\
  17 &      36  &     270 &$(-p,q,-r,s)   $ &    3  &      15  &     2   &  3   &     23   \\     \hline
  18 &      -17 &     420 &$(-p,q,-r,s)   $ &    3  &      5   &     2   &  14  &     24    \\
  18 &      -17 &     420 &$K_p + CS_{q,r}$ &    8  &      12  &     5   &      &     25    \\    \hline
  18 &      35  &     240 &$(-p,q,r)      $ &    4  &      4   &     16  &      &     24    \\
  18 &      35  &     240 &$K_p + CS_{q,r}$ &    4  &      16  &     5   &      &     25    \\    \hline
  19 &      40  &     300 &$(-p,q,r)      $ &    4  &      5   &     16  &      &     25    \\
  19 &      40  &     300 &$(-p,q,r,-s)     $ &    2  &      5   &     15  &  2   &     24    \\     \hline
\end{tabular}
\end{tabular}
}
\end{table}
~\\
~

\end{document}